\documentclass[a4paper,10pt]{amsart}

\usepackage{amsmath,amssymb,amsthm,a4wide,graphicx,tikz}
\usetikzlibrary{shapes}
\usepackage{caption}     

\newtheorem{theorem}{Theorem}
\newtheorem*{thm}{Theorem}
\newtheorem{lemma}{Lemma}

\newcommand{\sgn}{\operatorname{sgn}}

\begin{document}
\title[]{The Geometry of Nodal sets and Outlier detection}

\author[]{Xiuyuan Cheng}
\address[Xiuyuan Cheng]{Applied Mathematics Program, Yale University, New Haven, CT 06510, USA}
\email{xiuyuan.cheng@yale.edu}

\author[]{Gal Mishne}
\address[Gal Mishne]{Applied Mathematics Program, Yale University, New Haven, CT 06510, USA}
\email{gal.mishne@yale.edu}

\author[]{Stefan Steinerberger}
\address[Stefan Steinerberger]{Department of Mathematics, Yale University, New Haven, CT 06510, USA}
\email{stefan.steinerberger@yale.edu}

\begin{abstract} Let $(M,g)$ be a compact manifold and let $-\Delta \phi_k = \lambda_k \phi_k$ be the sequence
of Laplacian eigenfunctions. We present a curious new phenomenon
which, so far, we only managed to understand in a few highly specialized cases: 
the family of functions $f_N:M \rightarrow \mathbb{R}_{\geq 0}$ 
$$ f_N(x) = \sum_{k \leq N}{ \frac{1}{\sqrt{\lambda_k}} \frac{|\phi_k(x)|}{\|\phi_k\|_{L^{\infty}(M)}}}$$
seems strangely suited for the detection of anomalous points on the manifold. It may be heuristically
interpreted as the sum over distances to the nearest nodal line and potentially hints at a new 
phenomenon in spectral geometry. We give rigorous statements on 
the unit square $[0,1]^2$ (where minima localize in $\mathbb{Q}^2$) and on Paley graphs (where $f_N$ recovers the
geometry of quadratic residues of the underlying finite field $\mathbb{F}_p$). Numerical examples show that the phenomenon seems to arise
on fairly generic manifolds.
\end{abstract}
\maketitle

\vspace{-20pt}

\section{Introduction. }
\subsection{Introduction.} The purpose of this paper is to report a curious observation in spectral geometry
that seems intrinsically interesting and may have nontrivial applications in outlier detection.
 Numerical examples on rough real-life data (see \S 3 below) indicate that the phenomenon is robust 
and seems to occur on fairly generic manifolds.

\begin{quote} \textbf{Observation.}  Let $(M,g)$ be a compact manifold and let $-\Delta \phi_k = \lambda_k \phi_k$ be the sequence
of Laplacian eigenfunctions. The maxima and minima of the function
\begin{equation}
\label{eq:anomaly}\nonumber
f_N(x) = \sum_{k \leq N}{ \frac{1}{\sqrt{\lambda_k}} \frac{|\phi_k(x)|}{\|\phi_k\|_{L^{\infty}(M)}}}
\end{equation}
seem to correspond to \textit{special} points on the manifold.
\end{quote}
The notion of \textit{special} point is vague and depends on the context: the special points
turn out to be the rational numbers on $[0,1]$, quadratic (non-)residues in finite fields $\mathbb{F}_p$ on Paley Graphs and sea-mines
in sonar data. We have no theoretical
understanding of the underlying phenomenon, nor do we understand its extent or the proper language in
which it should be phrased.

\subsection{Number Theory on $[0,1]$.} A first indicator that this quantity may
be of some interest was given by the third author \cite{stein} in the special case of the interval $[0,1]$. 
\begin{figure}[h!]
\begin{minipage}{0.49\textwidth}
\begin{center}
\begin{tikzpicture}[scale=0.85]
\node at (0,0) {\includegraphics[width= 0.8\textwidth]{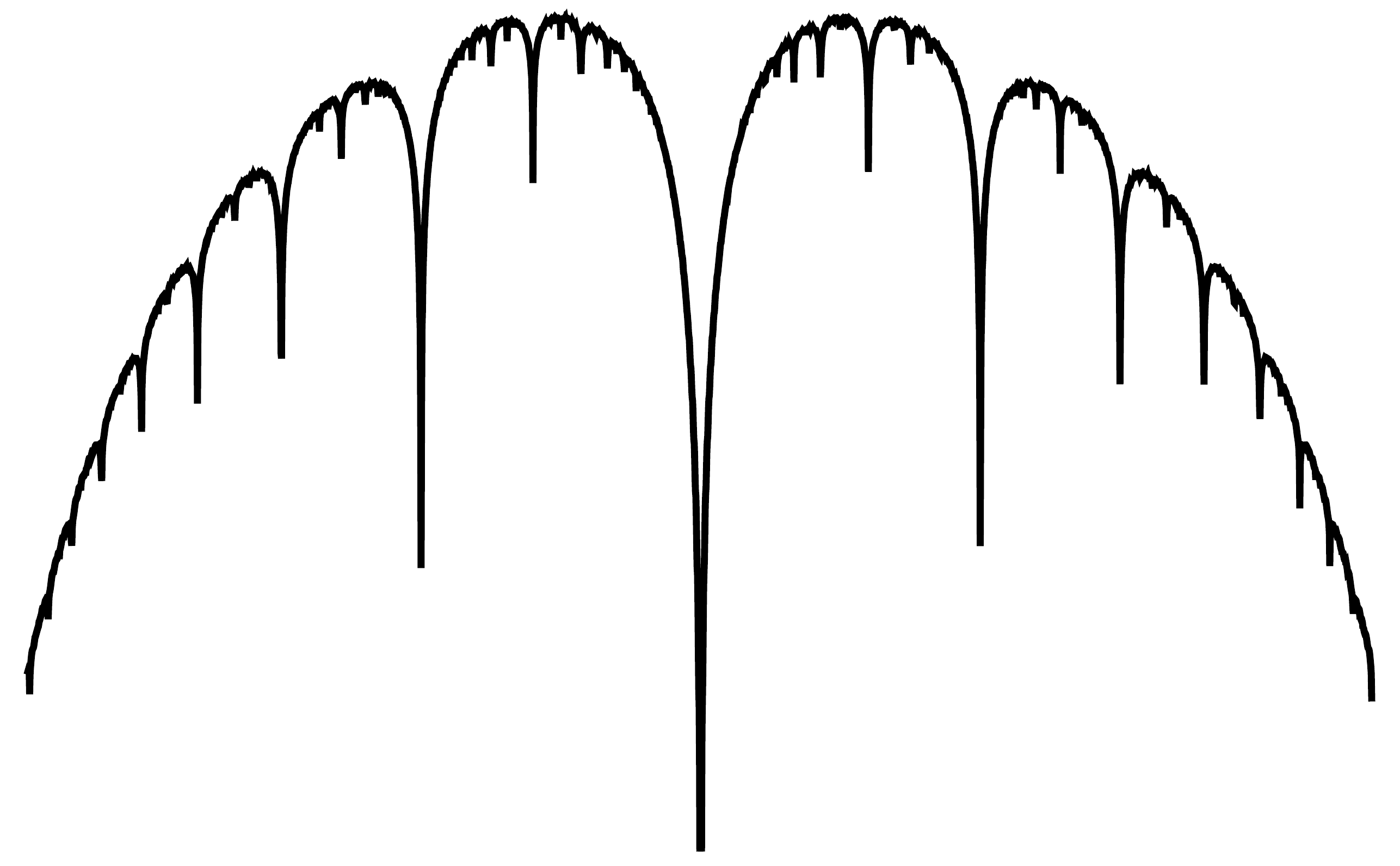}};
\draw [ultra thick] (-3,-2) -- (3,-2);
\draw [ultra thick] (-3,-2.1) -- (-3,-1.9);
\node at (-3, -2.4) {$0.1$};
\draw [ultra thick] (3,-2.1) -- (3,-1.9);
\node at (3, -2.4) {$0.9$};
\end{tikzpicture}
\end{center}
\end{minipage}
\begin{minipage}{0.49\textwidth}
\begin{center}
\begin{tikzpicture}[scale=0.85]
\node at (0,0) {\includegraphics[width= 0.8\textwidth]{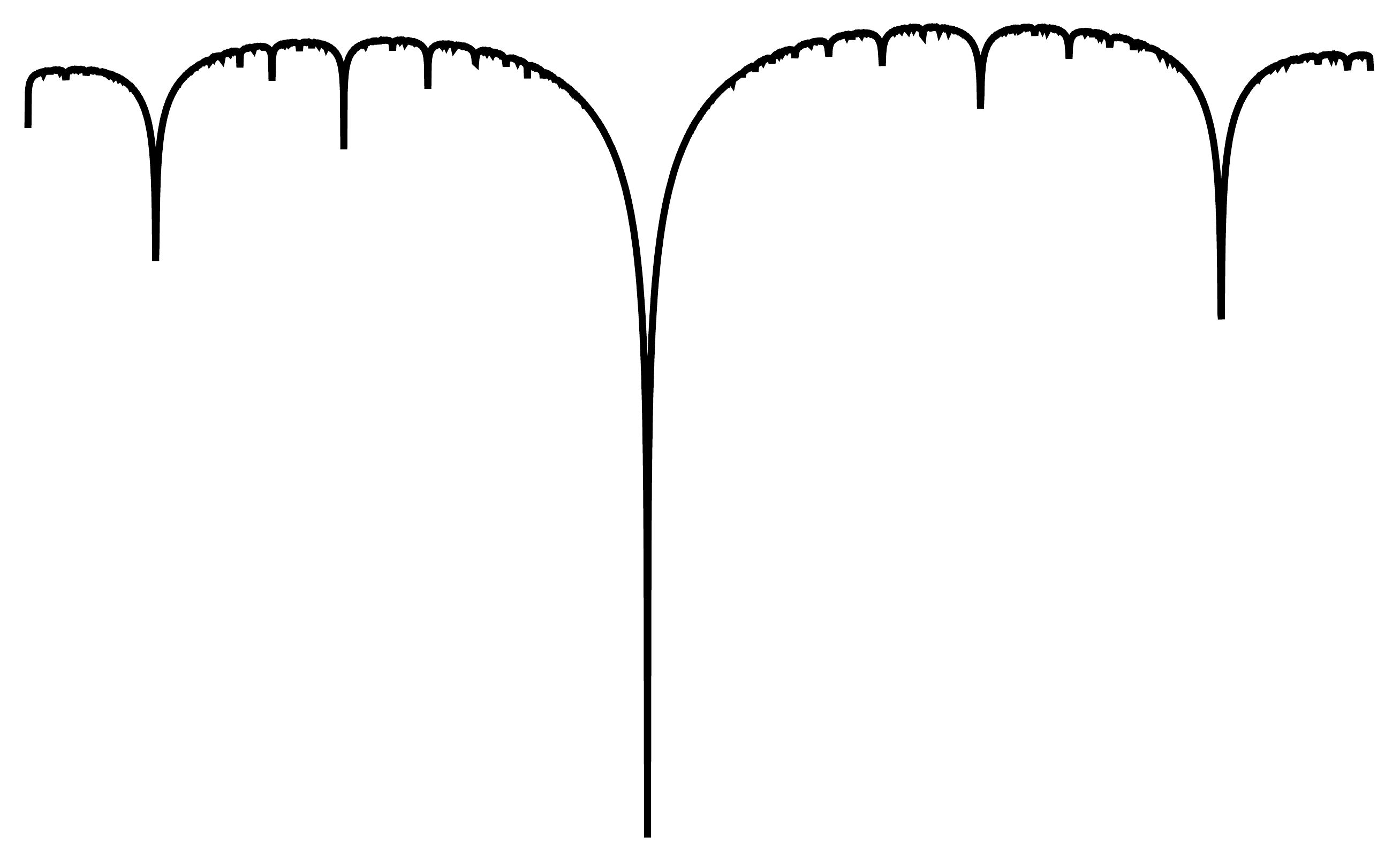}};
\draw [ultra thick] (-3,-2) -- (3,-2);
\draw [ultra thick] (-3,-2.1) -- (-3,-1.9);
\node at (-3, -2.4) {$0.38$};
\draw [ultra thick] (3,-2.1) -- (3,-1.9);
\node at (3, -2.4) {$0.39$};
\end{tikzpicture}
\end{center}
\end{minipage}
\captionsetup{width=0.9\textwidth}
\caption{The function $f_{N}$ for $N=50000$ on $[0.1, 0.9]$ and zoomed in (right): local minima are located
at rational numbers (the big cusp in the right is located at $x=5/13$).}
\end{figure}
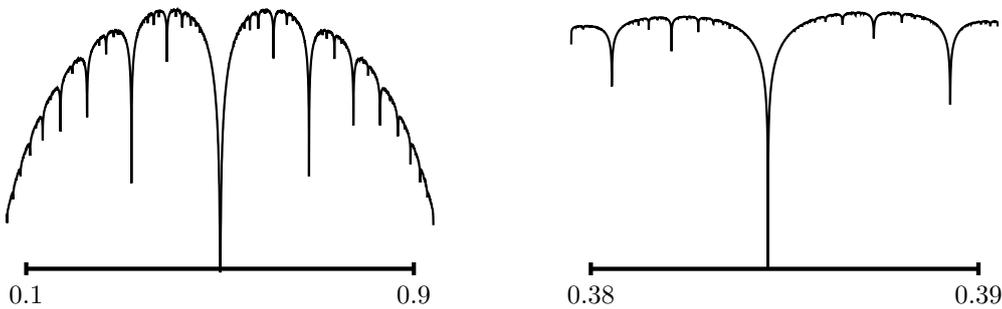

The eigenfunctions of the Laplacian on $[0,1]$ (with Dirichlet boundary conditions) are merely trigonometric functions; thus
$$ f_N(x) = \sum_{1 \leq k \leq N}{ \frac{1}{\sqrt{\lambda_k}} \frac{|\phi_k(x)|}{\|\phi_k\|_{L^{\infty}(M)}}} =\sum_{k=1}^{N}{ \frac{|\sin{k \pi x}|}{k} }.$$
This simple function captures a lot of information: it has strict local minima 
in the rational points.
\begin{thm}[S. 2016] Let $p,q \in \mathbb{Z}$ and $q \neq 0$. The function $f_N(x)$ has a strict local minimum in the point $x=p/q \in \mathbb{Q}$ for all $N \geq (1+o(1))q^2/\pi$.
\end{thm}
We should emphasize that there are many natural questions about $f_N$ in this simple setting that are still open: where are the local maxima? Is their location in any way related to reals that
cannot be well approximated by rationals with small denominator? Is there any connection to the continued fraction expansion?\\

This shows $f_N$ to be of interest on $[0,1]$, however, we found that it is indeed successful in isolating natural points of interest in a variety
of settings and the rest of the paper gives both theoretical and numerical examples. A first example is Fig.\ref{fig:mine1}, which shows a sea-mine
 in a side-scan sonar image with an arrow explicitly identifying the mine (left) and $f_{15}$ (right). Note
that the background is highly cluttered and the function $f_N$ performs astonishingly well in detecting the sea-mine.

\begin{figure}[h!]
\centering{\includegraphics[width=0.95\linewidth]{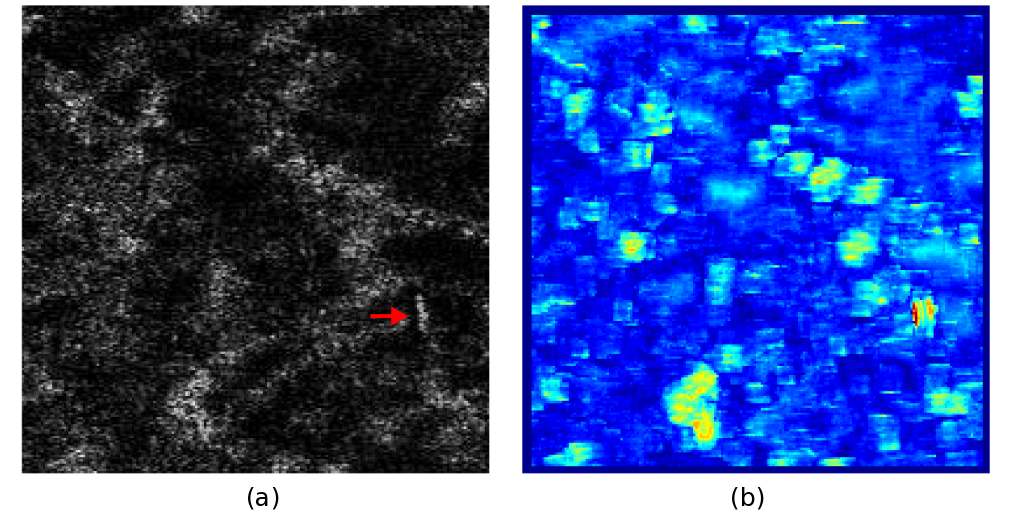}}
\caption{(a) Side scan sonar image with sea-mine indicated by red arrow. (b) The value of $f_{15}$ for each pixel (red being large values, blue being small values).
\label{fig:mine1}}
\end{figure}

At this point, we are not aware of any theory that could explain the behavior of $f_N$. The existing results (see \S 2 for the behavior of $f_N$ on Paley graphs)
seem to suggest that a natural starting point for its study might be that of simple manifolds ($\mathbb{T}$, $\mathbb{S}^2$, the unit disk $\mathbb{D}$ and
finite graphs equipped with the Graph Laplacian), where the behavior seems to be coupled to explicit number-theoretic problems such as the one discussed
above for $[0,1]$.

\subsection{Organization.} \S 2 gives three rigorous results that show $f_N$ to capture meaningful information in three
very different cases. \S 3 gives some explicit numerical examples and \S 4 gives proofs of the results. We conclude
with some remarks in \S 5.

\section{Three Rigorous Results}

\subsection{The Unit Square.}  We generalize the result above to the unit square $[0,1]^2$.

\begin{theorem} Let $M=[0,1]^2$ with the flat metric and consider 
$$ f_N(x,y) = \sum_{1 \leq k \leq N}{ \frac{1}{\sqrt{\lambda_k}} \frac{|\phi_k(x,y)|}{\|\phi_k\|_{L^{\infty}(M)}}}$$
For every fixed $0 < y < 1$, the function $f_N(\cdot, y)$ has local minimizers in $x=p/q$
for all $N$ sufficiently large and the same holds for $f_N(x, \cdot)$, fixed $0 < x < 1$ and $N$
sufficiently large. Moreover, $f_N(x,y)$ has strict local minimizer in $(p/q, r/s) \in \mathbb{Q}^2$ for $N$ sufficiently large.
\end{theorem}
There does not seem to be any major obstruction to generalizing the result to $[0,1]^d$ but, for the
sake of a clear and concise proof, we restrict ourselves to $d=2$.
It would be of interest to have similar results on other domains. Fine estimates on the Bessel function could possibly allow
to obtain an analogue of Theorem 1 on the unit disk $\mathbb{D} \subset \mathbb{R}^2$.
It is not clear whether many other domains would be similarly accessible via an explicit analysis
because eigenfunctions are rarely known in closed form. Other natural geometric objects are
ellipses and the equilateral triangle; an additional difficulty
on $\mathbb{S}^{d-1}$ is the presence of multiplicities (see \S 5.3). However, here it may nonetheless be very worthwhile
to investigate the behavior of the classical spherical harmonics as they present a natural choice of basis within each
eigenspace.

\subsection{Paley Graphs} While the number of classical domains with an explicit closed-form expression
for the eigenfunctions is not large, spectral graph theory provides a large number of graphs that,
equipped with the Graph Laplacian, become a natural object of investigation.
We proceed by considering a curious example provided by the Paley Graph. Let $p \equiv 1~$(mod 4) be prime and
consider a graph with vertices identified by elements of the finite field $V = \left\{0, 1, \dots, p-1\right\} \cong \mathbb{F}_p$, where a pair $(a,b) \in V \times V$ is connected by
an edge if and only if the equation
$$a-b \equiv x^2 \qquad \mbox{has a solution for}~0 \neq x \in \mathbb{F}_p.$$

\begin{figure}[h!]
\begin{center}
\includegraphics[width=0.3\textwidth]{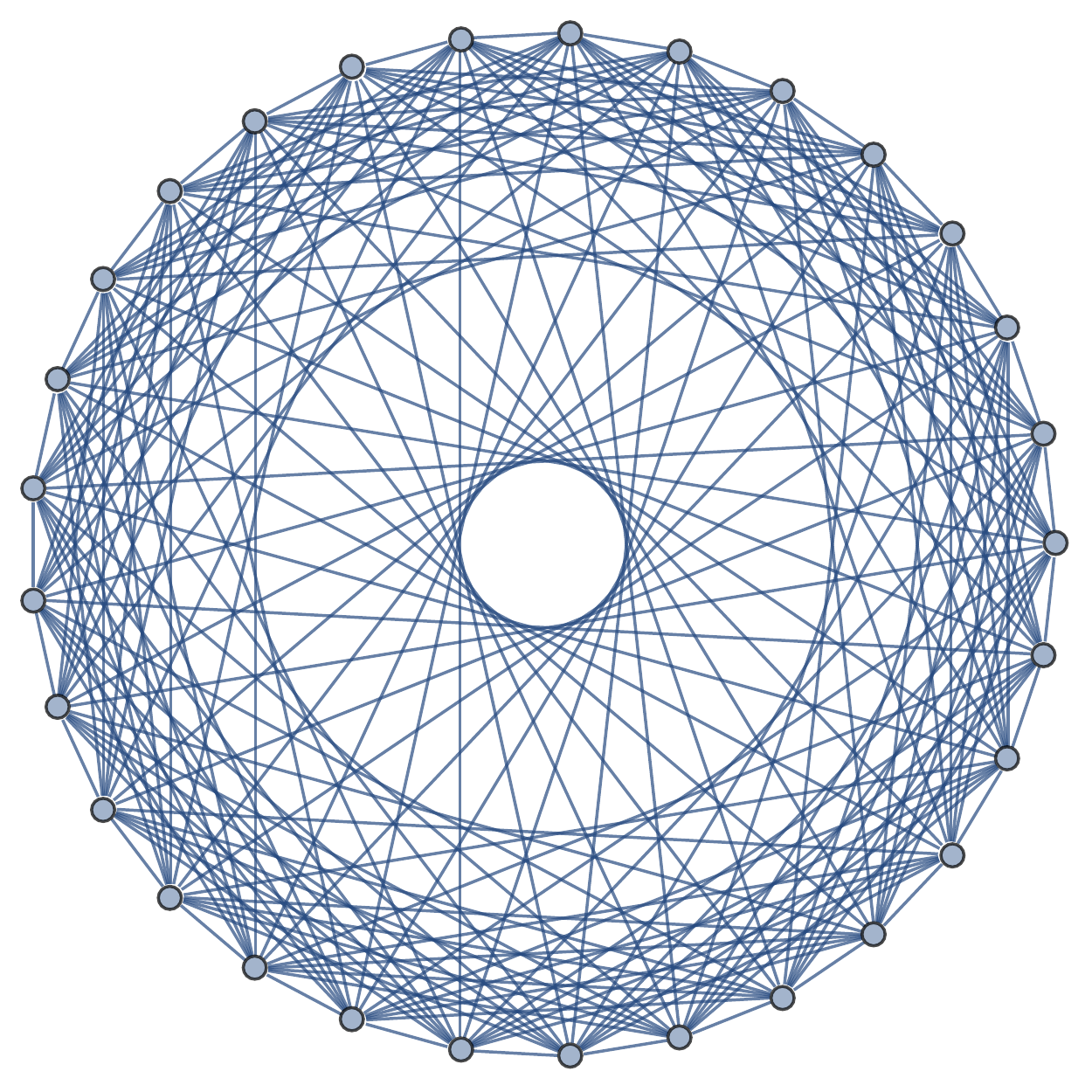}
\end{center}
\caption{The Paley Graph for $\mathbb{F}_{29}$.}
\end{figure}

An example of the arising graph for $\mathbb{F}_{29}$ can be seen in Figure 2. We will work with complex eigenvalues and slightly change $f_N$ to
$$ \qquad  \sum_{k}{ \frac{1}{\sqrt{\lambda_k}} \frac{\phi_k(x)}{\|\phi_k\|_{L^{\infty}(M)}}}$$
to avoid losing the information in the phase. Moreover, since the matrix is finite, it is more natural to sum over all the eigenspaces as opposed to introducing
an artificial cut-off: the natural analogue of $f_N$ on a finite graph is summation over all eigenvectors.
\begin{theorem} For the canonical choice of eigenfunctions of the Paley graph over $\mathbb{F}_p$, the function
$$  \sum_{k=0}^{p-1}{ \frac{1}{\sqrt{\lambda_k}} \frac{\phi_k(x)}{\|\phi_k\|_{L^{\infty}(M)}}}$$
assumes exactly three different values and depending on whether $x=0$, a quadratic residue or a nonquadratic residue in $\mathbb{F}_p$.
\end{theorem}
The three different values can be written down in closed form (see the proof for details).
This result is again of a number-theoretical flavor and shows that quantities of the type under consideration seem to recover relatively subtle details
about the underlying graph. Analogues of Theorem 2 are expected 
hold on other algebraic structures and graphs with underlying symmetries; a more thorough understanding of graphs could greatly aid the
understanding of the general case and we believe it to be of interest.

\subsection{$\mathbb{T}$ with a perturbed potential.} We will now consider the Laplacian on the one-dimensional $\mathbb{T}$ with a slight local perturbation of a constant potential in a point. More precisely, let $\phi:[0,1]\rightarrow \mathbb{R}_{<0}$ be a negative function
and let us consider the manifold $(\mathbb{T}, dx)$ with the potential
$$ V_{y,\varepsilon}(x) = \begin{cases} 1 + \varepsilon \phi\left( \frac{x-y}{\varepsilon} \right) \qquad &\mbox{if}~y \leq x \leq y+ \varepsilon \\ 1 \qquad &\mbox{otherwise.} \end{cases}$$ 
We would expect that the location $y$ of the perturbation is recovered as an `anomaly' that is discovered by
the eigenfunctions of $-\Delta_{} + V_{y, \varepsilon}$. 
\begin{theorem} For every $\varepsilon > 0$, there exists a $N_{\varepsilon} \in \mathbb{R}$ such that, for all $N \leq N_{\varepsilon}$
$$ \sum_{k \leq N}{ \frac{1}{\sqrt{\lambda_k}} \frac{|\phi_k(x)|}{\|\phi_k\|_{L^{\infty}(M)}}} \qquad \mbox{has a strict local minimum in}~y \leq x_0 \leq y+\varepsilon.$$
Moreover, we have that $N_{\varepsilon} \rightarrow \infty$ as $\varepsilon \rightarrow 0$.
\end{theorem}
The result states that the anomaly is indeed discovered independently of the cutoff as long as that cutoff is not too large. The proof is based on a bifurcation of the eigenspace. Understanding the behavior for $\varepsilon$
fixed and $N \rightarrow \infty$ seems more challenging but potentially quite interesting.

\section{Experimental Results}
\subsection{Sea-mine detection}
\label{sec:seamine}
We calculate $f_{15}$ for the problem of detecting sea-mines in side-scan sonar images, collected by the Naval Surface Warfare Center Coastal System Station.

\begin{figure}[h!]
\centering{\includegraphics[width=0.82\linewidth]{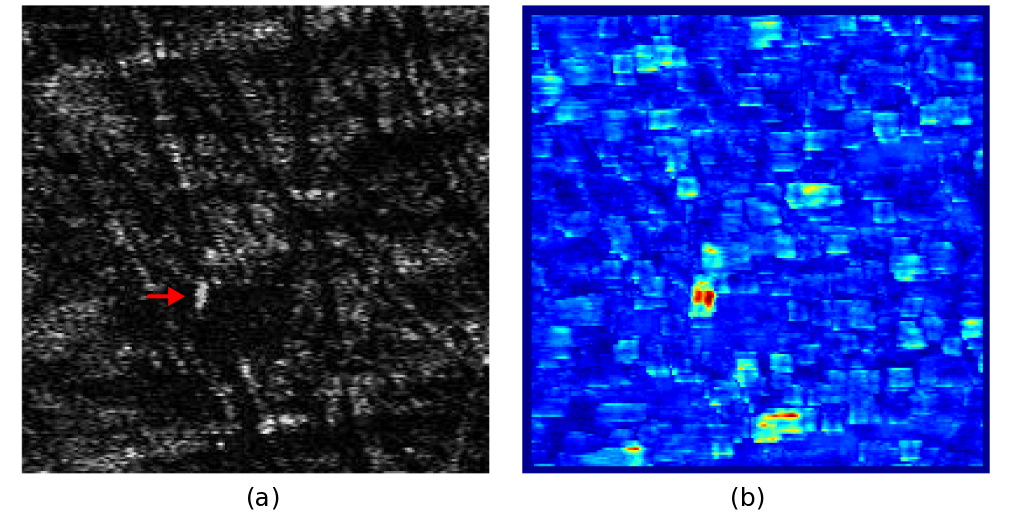}}
\caption{(a) Side scan sonar image with sea-mine indicated by red arrow. (b) The value of $f_{15}$ for each pixel (red being large values, blue being small values).
\label{fig:mine4}}
\end{figure}

Two such images are presented in Fig.~\ref{fig:mine1}(a) and Fig.~\ref{fig:mine4}(a).
Automatic detection of sea-mines in side-scan sonar imagery is difficult because of to the high variability in the appearance of the target and sea-bed reverberations (background clutter).
Fig.~\ref{fig:mine1}(b) and Fig.~\ref{fig:mine4}(b) show $f_{15}$ for these images; it achieves a maximal value on the sea-mine, thus enabling its detection.
The images were cropped to a region sized 200 (range)$\times$ 200 (cross-range) cells which contains a sea-mine and various background types. 
We construct the graph by representing each pixel by its surrounding $8 \times 8$ patch. The weight between patches is based on a Gaussian kernel of the Euclidean distance. 
For efficient computation we connect each patch to only its 16 nearest neighbors.

\subsection{Defect detection}
Figure~\ref{fig:wafer} shows the value of $f_{20}$ on a Scanning Electron Microscope (SEM) image of a patterned wafer (image source: Applied Materials, Inc.). We consider the defect in the image to be an anomaly while the patterned wafer is considered normal background clutter.
The parameters of the images and graph construction are as for the images in Sec.~\ref{sec:seamine}.
The maximal value of $f_N$ is indeed attained on the defect, which is barely discernible in the original image.

\begin{figure}[h!]
\centering{\includegraphics[width=0.8\linewidth]{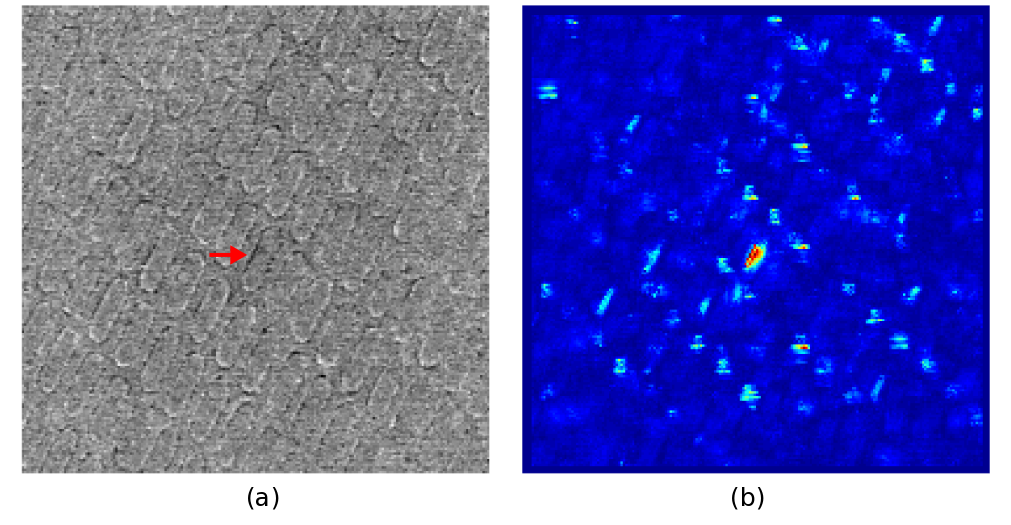}}
\caption{(a) SEM wafer image with defect indicated by red arrow. (b) The value of $f_{20}$ for each pixel (red being large values, blue being small values).
\label{fig:wafer}}
\end{figure}

\subsection{Graphics / 3D mesh}
Figure~\ref{fig:bunny} shows the value of $f_{30}$ on the Stanford bunny (a) as well as a suitable quantization (b). We observe
again that the maxima are assumed on the `outliers of the bunny' (tips of the ears, nose, paws and tail) whereas minima are assumed in the `center'.

\begin{figure}[h!]
\centering{\includegraphics[width=0.6\linewidth]{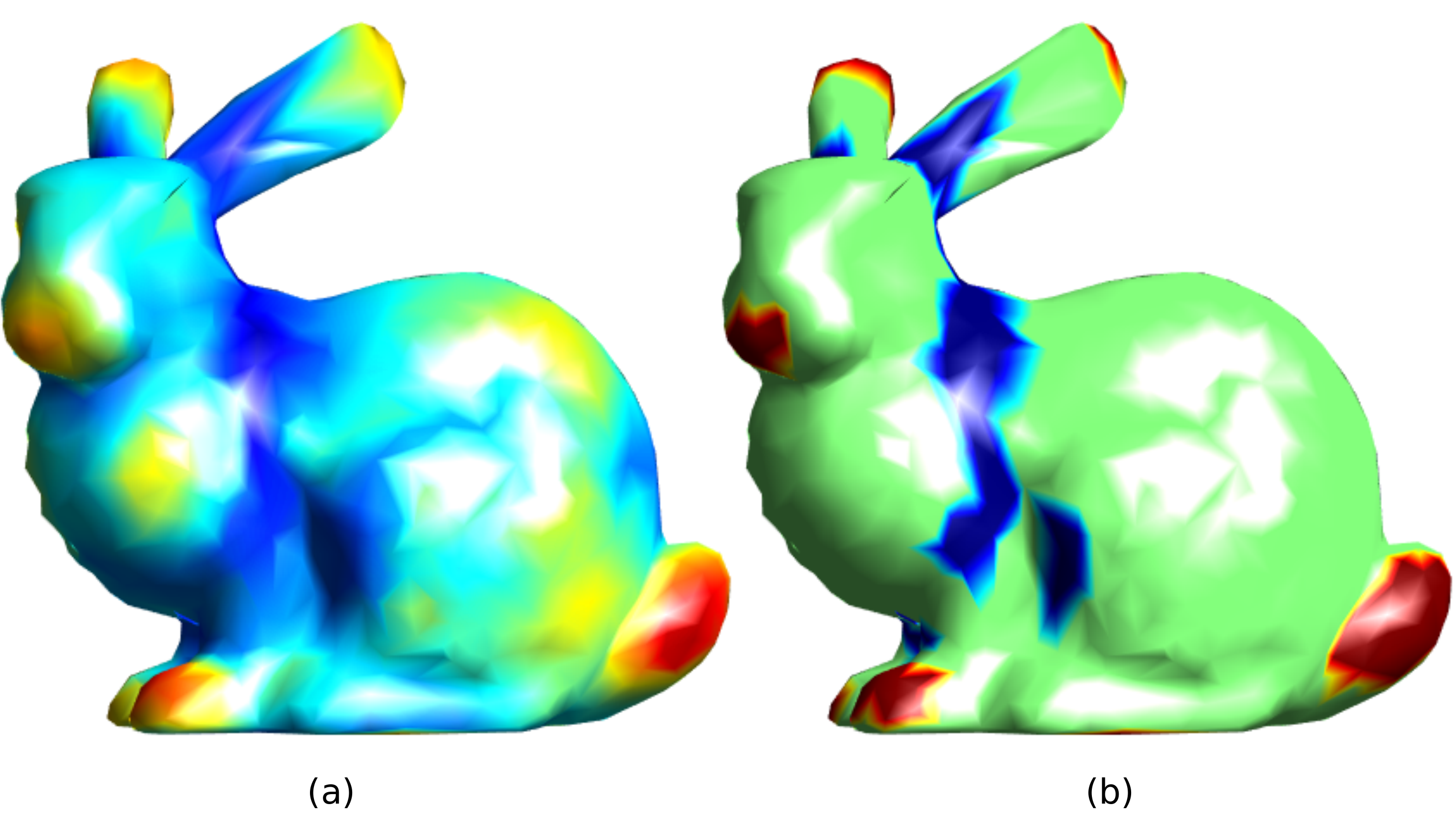}}
\caption{(a) $f_{30}$ on Stanford bunny (b) Quantized anomaly score: red is 90th percentile and blue is 10th percentile.
\label{fig:bunny}}
\end{figure}

Figure \ref{fig:homer} shows $f_{30}$ on a popular cartoon figure; we discover that $f_{30}$ is maximal at fingers and toes and the nose (and minimal at the chest). As is perhaps not surprising, it is difficult to turn these examples into a rigorous statements but they certainly underline
our main point of the function assuming minima at central locations and maxima at outliers.

\begin{figure}[h!]
\centering{\includegraphics[width=0.6\linewidth]{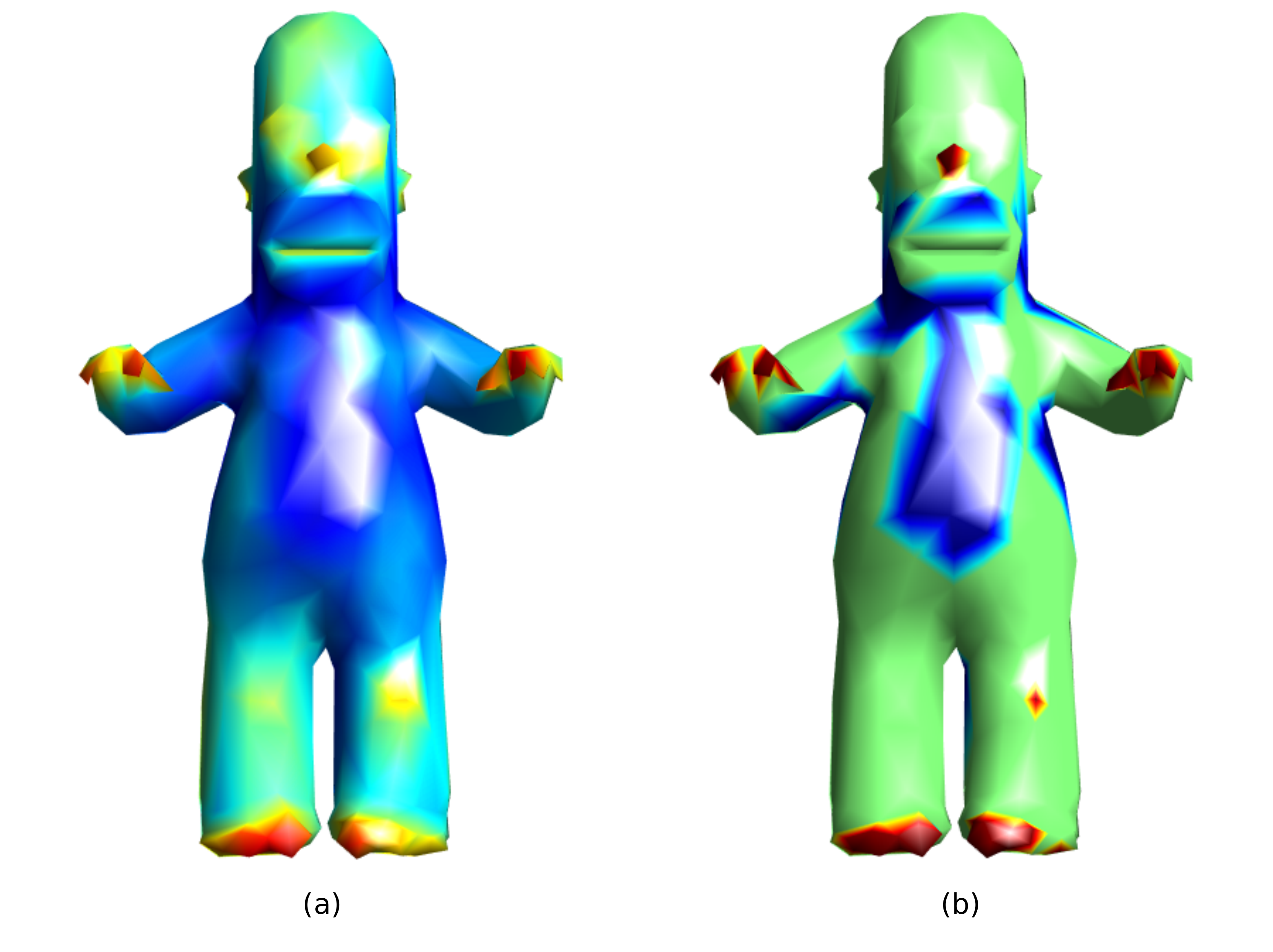}}
\caption{(a) $f_{30}$ on Homer mesh (b) Quantized anomaly score: red is 90th percentile and blue is 10th percentile).
\label{fig:homer}}
\end{figure}

\vspace{-0pt}

\section{Proofs}
\subsection{Proof of Theorem 1.}
\begin{lemma} \label{lem} Let $(a_n)_{n \geq 1}$ be a periodic sequence satisfying $a_{n+p} = a_n$ and having vanishing mean value over one period and let $(b_n)$ be positive and monotonically increasing. Then, for all $N \geq 1$,
$$ \left| \sum_{n=1}^{N}{ a_n b_n} \right| \leq    \frac{ 3b_N}{2}  \sum_{n=1}^{p}{| a_n|}.$$
\end{lemma}
\begin{proof}  We first show a slightly different statement: if the sequence $(b_n)$ is positive and monotonically decreasing, then for all $N \geq 1$
$$ \left| \sum_{n=1}^{N}{ a_n b_n} \right| \leq b_1\sum_{n=1}^{p}{|a_n|}.$$
This can be seen as follows: the classical rearrangement inequality (see e.g. \cite{hardy}) implies that, for all $m\geq 1$
$$ \sum_{k = mp + 1}^{mp + p}{a_m b_m} \leq  \sum_{k = mp + 1}^{mp + p}{a^*_m b_m},$$
where $a^*_m$ is the rearrangement of the $p$ elements that occur within one period of the sequence $a$ in monotonically decreasing order. Then, however, bounding the series from above and regrouping the terms we obtain
an alternating series and the classical Leibnitz argument for alternating series applies. The simple example $a=(1,-1,1,-1,\dots)$ and $b_n=1$ shows the inequality to be sharp.
The statement now follows from writing
\begin{align*}  \left| \sum_{n=1}^{N}{ a_n b_n} \right| &= \left| \sum_{n=1}^{N}{ a_n b_N} - \sum_{n=1}^{N}{ a_n (b_N-b_n)} \right|\\
&\leq  \left| \sum_{n=1}^{N}{ a_n b_N}\right| +  \left| \sum_{n=1}^{N}{ a_n (b_N-b_n)} \right|.
\end{align*}
The mean zero condition on the sequence $(a_n)$ implies that
\begin{align*}
   \left| \sum_{n=1}^{N}{ a_n b_N}\right| &= b_N   \left| \sum_{n=1}^{N}{ a_n}\right| \leq   b_N    \max_{1 \leq \sigma \leq p} \left| \sum_{n=1}^{\sigma}{ a_n}\right| \leq \frac{b_N}{2} \sum_{n=1}^{p}{|a_n|}.
\end{align*}
The second sum can be estimate by using the statement above
$$  \left| \sum_{n=1}^{N}{ a_n (b_N-b_n)} \right| \leq b_N \sum_{n=1}^{p}{ |a_n|}$$
and this implies the result. 
\end{proof}

\begin{lemma} Let $0 < y < 1$. Then,
$$ \lim_{n \rightarrow \infty} \frac{1}{n} \sum_{k = 1}^{n}{ | \sin{k \pi y} |}  \geq \frac{1}{2}$$
with equality if and only if $y=1/2$.
\end{lemma}
\begin{proof}[Sketch of the Proof.]  The problem reduces to understanding the behavior of $\left(ky - \left\lfloor ky \right\rfloor\right)_{k=1}^{\infty}$ on the unit interval. The asymptotic behavior is completely obvious: if $y$ is irrational,
then the sequence is uniformly distributed and
$$ \lim_{n \rightarrow \infty}{ \frac{1}{n}  \sum_{k = 1}^{n}{ | \sin{k \pi y} |}  } = \int_{0}^{1}{|\sin{\pi x}| dx} = \frac{2}{\pi}.$$
If $y = p/q \in \mathbb{Q}$ with $\gcd(p,q) = 1$, then
$$ \lim_{n \rightarrow \infty}{ \frac{1}{n}  \sum_{k = 1}^{n}{ | \sin{k \pi y} |}  } = \frac{1}{q} \sum_{k=0}^{q-1}{ \left| \sin{\left(\frac{k \pi}{q}\right)} \right| },$$
which by a simple concavity argument (i.e. Jensen's inequality) can be shown to be monotonically increasing in $q$. The quantity is therefore minimal for $y=1/2$, which then immediately leads to a uniform estimate for all $0 < y < 1$
$$ \lim_{n \rightarrow \infty}{ \frac{1}{n}  \sum_{k = 1}^{n}{ | \sin{k \pi y} |}  } \geq \frac{1}{2}.$$
\end{proof}

\begin{proof}[Proof of Theorem 1]
The quantity to be evaluated on the unit square is
$$ f(x,y) = \sum_{m^2 + n^2 \leq \lambda}{ \frac{|\sin{(m \pi x)}|| \sin {(n \pi y)}|}{\sqrt{m^2+n^2}}}.$$
We simplify
$$ f(x+\varepsilon, y) - f(x,y) =  \sum_{m^2 + n^2 \leq \lambda}{ \frac{| \sin {(n \pi y)}|}{\sqrt{m^2+n^2}}\left( |\sin{(m \pi (x+\varepsilon))}| - |\sin{(m \pi x)}| \right)  }$$
 and observe that for $x \in \mathbb{R}$ and $\varepsilon \rightarrow 0$
$$ |\sin{(x+\varepsilon)}| - |\sin{(x)}| = \begin{cases} |\varepsilon| + \mathcal{O}(\varepsilon^2) \qquad &\mbox{if}~x/\pi \in \mathbb{Z} \\
\varepsilon\sgn(\sin{(x)}) \cos{(x)} + \mathcal{O}(\varepsilon^2)  \qquad &\mbox{otherwise.} \end{cases}$$
Here, $\mbox{sgn}$ denotes the signum function
$$ \mbox{sgn}(x) = \begin{cases} 1 \qquad &\mbox{if}~x > 0 \\
0\qquad &\mbox{if}~x = 0\\
-1 \qquad &\mbox{if}~x < 0.
\end{cases}$$
This leads us to the first order expansion (ignoring terms of size $\varepsilon^2$ and smaller)
\begin{align*}
 f(x+\varepsilon, y) - f(x,y)  &=  \sum_{m^2 + n^2 \leq \lambda \atop mx \in \mathbb{Z}}{ \frac{| \sin {(n \pi y)}|}{\sqrt{m^2+n^2}} m |\varepsilon|   } \\
&+  \sum_{m^2 + n^2 \leq \lambda \atop mx \notin \mathbb{Z}}{ \frac{| \sin {(n \pi y)}|}{\sqrt{m^2+n^2}} m \varepsilon \sgn(\sin{(m \pi x)}) \cos{(m \pi x)}   }
\end{align*}
Let us now assume that $x=p/q$ for some coprime $p,q \in \mathbb{Z}$ with $q \neq 0$. We will now show that, for $\lambda$ sufficiently large, the first term dominates the second term. The
symmetry $f(x,y) = f(y,x)$ will then imply the result. We start with the first term
\begin{align*}
  \sum_{m^2 + n^2 \leq \lambda \atop mx \in \mathbb{Z}}{ \frac{| \sin {(n \pi y)}|}{\sqrt{m^2+n^2}} m |\varepsilon|   }= |\varepsilon|  \sum_{m^2 + n^2 \leq \lambda \atop q|m}{ \frac{| \sin {(n \pi y)}|}{\sqrt{1+(n/m)^2}}},
\end{align*}
restrict the sum to a subset for which the denominator is small,
$$  \sum_{m^2 + n^2 \leq \lambda \atop q|m}{ \frac{| \sin {(n \pi y)}|}{\sqrt{1+(n/m)^2}}}    \geq     \frac{1}{\sqrt{2}}\sum_{m^2 + n^2 \leq \lambda \atop q|m \wedge m \geq n}{ | \sin {(n \pi y)}|},$$
and then bound it as
$$    \sum_{m^2 + n^2 \leq \lambda \atop q|m \wedge m \geq n}{ | \sin {(n \pi y)}|}  \geq    \sum_{\sqrt{\lambda}/2 \leq m \leq 99\sqrt{\lambda}/100 \atop q|m }{ \sum_{n \leq \sqrt{\lambda-m^2}}  | \sin {(n \pi y)}|}.$$
For this choice of parameters $\sqrt{\lambda}/2 \leq m \leq 99\sqrt{\lambda}/100$, we have that $\sqrt{\lambda - m^2} \gtrsim \sqrt{\lambda}$ and thus, as $\lambda$ increases, we may invoke
Lemma 2 to conclude that for $\lambda$ sufficiently large,
$$ \sum_{n \leq \sqrt{\lambda-m^2}} { | \sin {(n \pi y)}|} \geq \frac{\sqrt{\lambda-m^2}}{4}.$$
This implies that, for $\lambda$ sufficiently large,
 $$   \frac{1}{\sqrt{2}}\sum_{m^2 + n^2 \leq \lambda \atop q|m \wedge m \geq n}{ | \sin {(n \pi y)}|}   \geq    \frac{1}{\sqrt{2}}\sum_{\sqrt{\lambda}/2 \leq m \leq 99\sqrt{\lambda}/100 \atop q|m }{ \frac{\sqrt{\lambda-m^2}}{4}}.$$
A simple comparison shows that for $\lambda$ sufficiently large and up to errors of lower order
$$ \sum_{\sqrt{\lambda}/2 \leq m \leq 99/100\sqrt{\lambda} \atop q|m }{ \frac{\sqrt{\lambda-m^2}}{4}} \sim \frac{1}{q} \sum_{\sqrt{\lambda}/2 \leq m \leq 99\sqrt{\lambda}/100}{ \frac{\sqrt{\lambda-m^2}}{4}} \sim \frac{1}{4q} \int_{\sqrt{\lambda}/2}^{99\sqrt{\lambda}/100}{ \sqrt{\lambda - x^2}dx} \gtrsim \frac{\lambda}{16 q}.$$
This shows that the first sum is growing at least linearly $\lambda$, it remains to show that it dominates the second sum as $\lambda$ gets large.
We start by rewriting it as
$$ \varepsilon \sum_{n=1}^{\lfloor \sqrt{\lambda} \rfloor}{ | \sin {(n \pi y)}| \sum_{m \leq \sqrt{\lambda - n^2}}{  \frac{ \sgn(\sin{(m \pi x)}) \cos{(m \pi x)}}{\sqrt{1+(n/m)^2}}}}$$
The function $\sgn(\sin{(x)}) \cos{(x)}$ has period $\pi$ and the map $k \rightarrow k \cdot p$ is a permutation on $\mathbb{Z}_q$. It is then easy to see that the symmetries of sine and cosine imply
$$\sum_{k=1}^{q}{ \sgn\left(\sin{ \left( \frac{k \pi p}{q}\right)}\right) \cos{ \left( \frac{k \pi p}{q}\right)}  } = 0.$$
We apply Lemma~\ref{lem} with the $q-$periodic sequence
$$ a_m  = \sgn\left(\sin{ \left( \frac{m \pi p}{q}\right)}\right) \cos{ \left( \frac{m \pi p}{q}\right)} \qquad \mbox{and} \qquad b_m= \frac{1}{\sqrt{1+(n/m)^2}}$$
and obtain
\begin{align*}
  \sum_{m \leq \sqrt{\lambda - n^2}}{  \frac{ \sgn(\sin{(m \pi x)}) \cos{(m \pi x)}}{\sqrt{1+(n/m)^2}}} \leq  \frac{3}{2} b_{\left\lfloor \sqrt{\lambda-n^2} \right\rfloor} q \leq 2 q.
\end{align*}
A straightforward summation then yields
$$  \varepsilon \sum_{n=1}^{\lfloor \sqrt{\lambda} \rfloor}{ | \sin {(n \pi y)}| \sum_{m \leq \sqrt{\lambda - n^2}}{  \frac{ \sgn(\sin{(m \pi x)}) \cos{(m \pi x)}}{\sqrt{1+(n/m)^2}}}} \leq  2 |\varepsilon|  \sqrt{\lambda} q,$$
which has a smaller order of growth than the dominant term. This implies the result.
\end{proof}

\textit{Remark.} The proof also shows that one can expect $$\lambda \gtrsim \frac{q^4}{\min(y, 1-y)}$$
to suffice: this follows easily if one replaces the Lemma 2 by the following statement: for all $0 < y < 1$ and all
$n \geq 10/\min(y, 1-y)$, we have
$$ \sum_{k=1}^{n}{|\sin{k \pi y}|} \geq \frac{n}{100},$$
which follows quickly from simple arguments about the dynamical behavior of $ky - \left\lfloor ky \right\rfloor$.

\subsection{Proof of Theorem 2.}
\begin{proof} It is well-known that the Paley graph generated by $\mathbb{F}_p$ (when $p \equiv 1$ (mod 4)) has eigenvalues $(p-1)/2$ (with multiplicity 1) and eigenvalues 
$$ \lambda_{-} = \frac{p-\sqrt{p}}{2} \qquad \mbox{and} \qquad \lambda_{+} = \frac{p+\sqrt{p}}{2},$$
each with multiplicity $(p-1)/2$. The $k-$th eigenvector is given by
$$ e_k = \left( e^{ \frac{2\pi i j k}{p}}\right)_{j=0}^{p-1}.$$
Moreover, the eigenvalue associated to $e_k$ is $\lambda_+$ if and only if $k \equiv x^2$ (mod $p$) does not have a solution. Let us now fix a value $j \in \left\{1, \dots, p-1\right\}$. Then
$$ \sum_{k }{ \frac{\phi_k(j)}{\sqrt{\lambda_k}}} = \sum_{k \not\equiv x^2}{\frac{e^{ \frac{2\pi i j k}{p}}}{\sqrt{\lambda_{+}}}} + \sum_{k \equiv x^2}{\frac{e^{ \frac{2\pi i j k}{p}}}{\sqrt{\lambda_{-}}}}.$$
We now aim to show that this sum does not depend on the value of $j$ but only on whether $j$ is a quadratic residue or a quadratic non-residue. Suppose $j$ is a quadratic residue. The product of 
a quadratic residue and a non-residue is a non-residue; moreover, multiplication with a fixed non-zero number is a bijection in a finite field and therefore
 $$ \sum_{k \not\equiv x^2}{\frac{e^{ \frac{2\pi i j k}{p}}}{\sqrt{\lambda_{+}}}} =  \frac{1}{\sqrt{\lambda_{+}}} \sum_{k \not\equiv x^2}{e^{ \frac{2\pi i k}{p}}} \quad \mbox{and} \quad \sum_{k \equiv x^2}{\frac{e^{ \frac{2\pi i j k}{p}}}{\sqrt{\lambda_{-}}}} =   \frac{1}{\sqrt{\lambda_{-}}} \sum_{k \equiv x^2}{e^{ \frac{2\pi i k}{p}}}$$
and both sums are independent of $j$.
If $j$ is a quadratic non-residue, we can use the fact that the product of two non-residues is a residue and argue in exactly the same way.
\end{proof}

\subsection{Proof of Theorem 3.}
\begin{proof} The argument is a fairly straightforward perturbation argument. As $\varepsilon \rightarrow 0$, classical methods imply that different eigenspaces separated by a gap remain stable. The eigenspaces, all of which have multiplicity 2, undergo a bifurcation. An elementary computation shows that the eigenspace $\left\{\sin{(k x)}, \cos{(kx)} \right\}$ associated
to the eigenvalue $k^2+1$ of the unperturbed problem, splits into two eigenvalues
$$\lambda_{-} \sim k^2  + 1- c_1\varepsilon^2 \qquad \mbox{and} \qquad \lambda_{+}  \sim k^2 + 1 + c_2 \varepsilon^2$$
for some constants $c_1,c_2 > 0$. Moreover, the eigenfunctions will be, up to small error,
$$ \phi_{-} \sim \sin{( k(x - y) + O(\varepsilon))} \qquad \mbox{and} \qquad  \phi_{+} \sim \cos{( k(x - y) + O(\varepsilon))}.$$
Altogether, as $\varepsilon \rightarrow 0$, the sum over the eigenfunctions is approximated by
$$ \sum_{k \leq N}{ \frac{1}{\sqrt{\lambda_k}} \frac{|\phi_k(x)|}{\|\phi_k\|_{L^{\infty}(M)}}} \sim \sum_{k \leq N}{ \frac{|\sin{(k(x-y+O(\varepsilon)))}|}{k} + \frac{|\cos{(k(x-y+O(\varepsilon)))}|}{k}}$$
and the result follows from the elementary inequality
$$ |\sin{z}| + |\cos{z}| \geq 1 \qquad \mbox{with equality if and only if}~(2z)/\pi \in \mathbb{Z}.$$
Since the accuracy of these estimates increases as $\varepsilon \rightarrow 0$, we get that $N_{\varepsilon} \rightarrow \infty$ as $\varepsilon \rightarrow 0$.
\end{proof}

\section{Concluding Comments and Remarks}

\subsection{Heuristics.} In this section we describe how the phenomenon was originally discovered: the main insight
was that the $L^{\infty}-$normalized eigenfunction may, when properly rescaled, be interpreted as an
approximation of the distance to the nearest nodal line (as it appeared in the elliptic estimates in \cite{cheng, manas}).
More precisely, the starting point of this investigation was the following recent inequality for real-valued functions
due to M. Rachh and the third author \cite{manas}.
\begin{thm}[Rachh and S., 2016]
There is $c>0$ such that for all simply-connected $\Omega \subset \mathbb{R}^2$ and all $u:\Omega \rightarrow \mathbb{R}$, the following holds: if $u$ vanishes on the boundary $\partial \Omega$ and $|u(x_0)| = \| u \|_{L^{\infty}(\mathbb{R})}$, then
$$   \inf_{y \in \partial \Omega}{ \| x_0 - y\|}  \geq c \left\| \frac{\Delta u}{u} \right\|^{-1/2}_{L^{\infty}(\Omega)}.$$
\end{thm}

This inequality has a particular simple interpretation whenever $u$ is a Laplacian eigenfunction $-\Delta u = \lambda u$ since it implies
that a Laplacian eigenfunction assumes its maximum at least a wavelength ($\sim \lambda^{-1/2}$) away from its nodal domain $\left\{ x \in M: u(x) = 0\right\}$ on two-dimensional
domains.
The inequality in \cite{manas} also generalizes to higher dimensions once the notion of distance has been suitably adapted; a version on graphs
equipped with the Graph Laplacian due to Rachh and the first and third author \cite{cheng} gave some numerical evidence that the maxima and minima
of eigenfunctions correspond to points that are either very central or very decentralized. 
All these results combined suggests the vague heuristic that at least for generic eigenfunctions, $-\Delta u = \lambda u$,
$$ \frac{1}{\sqrt{\lambda}} \frac{|u(x)|}{\|u\|_{L^{\infty}(M)}} \sim \mbox{distance of}~x~\mbox{to the nodal set.}$$
This cannot be true in the sense of a mathematical theorem and it is easy to construct counterexamples (however, the left-hand side is always dominated by the right-hand side
if distance is replaced by a notion of distance based on capacity, see \cite{manas}); however, it should be `generically' true 
in all the natural ways: for example, we would expect it to be true with high likelihood in the random
wave model or in the setting of random eigenfunctions on geometries where eigenvalues have high multiplicity (i.e.
$\mathbb{S}^{d-1}$ or $\mathbb{T}^d$).

\subsection{Geometric analogues} A natural geometric quantity on the manifold
is then, for a fixed point $x \in M$, the sum of the distances to the nearest nodal set across multiple eigenfunctions -- since the
distance itself may be fairly nontrivial to compute and since this normalized quantity does seem to be a good indicator
of the actual distance, this suggests to consider the quantity
$$ \sum_{k \leq N}{ \frac{1}{\sqrt{\lambda_k}} \frac{|\phi_k(x)|}{\|\phi_k\|_{L^{\infty}(M)}}}$$
and this is how the phenomenon was discovered. Both the presence of the absolute value as well as the somewhat uncommon normalization in $L^{\infty}(M)$ 
are fairly unusual and we are not aware of any theory that would imply nontrivial statements for this quantity.
Clearly, this heuristic motivation raises another question.
\begin{quote}
\textbf{Question.} Is the purely geometric quantity, the sum over distances to the nearest nodal line, of comparable
intrinsic interest?
\end{quote}
In the examples that we consider, both quantities are roughly comparable. Moreover, the geometric quantity is not
as easy to define on a graph (since a graph generically does not have a nodal set but merely sign changes); nonetheless,
it could be an interesting avenue to pursue.

\subsection{Multiplicity of eigenvalues.}
We note that if there is an eigenvalue with multiplicity, then the quantity
$$ f_N(x) = \sum_{k \leq N}{ \frac{1}{\sqrt{\lambda_k}} \frac{|\phi_k(x)|}{\|\phi_k\|_{L^{\infty}(M)}}}$$
is not well-defined since the eigenfunctions $\phi_k$ are only defined up to a change of basis. Clearly, this
phenomenon does not generically arise in real-life situations; one natural way around would be to define
it as an integral over all rotations of the eigenspace but, ultimately, we do not understand the phenomenon
well enough to have any insight into this degenerate case.

\subsection{The $L^{\infty}-$norm} The normalization in $L^{\infty}(M)$ is certainly unusual; one could naturally
normalize in other $L^p-$spaces and, again in generic cases, the difference is marginal since we would expect
that $\| \phi_k\|_{L^{\infty}} \lesssim_{\varepsilon} \lambda_k^{\varepsilon} \|\phi_k\|_{L^2}$ in the generic 
quantum-ergodic case. The normalization in $L^{\infty}$ is motivated by the result above and seems
natural in the examples that we consider, however, other normalizations are conceivable.

\end{document}